\documentclass{article}
\usepackage[papersize={210truemm,297truemm},margin=3.0truecm,nohead]{geometry}
\usepackage{amsthm}
\usepackage{mathtools}
\usepackage{color}	%introduced to modify the text
\usepackage{ulem}	%introduced to modify the text
\theoremstyle{definition}
\newtheorem{theorem}{Theorem}[section]
\newtheorem{proposition}[theorem]{Proposition}
\newtheorem{lemma}[theorem]{Lemma}

\newtheorem{remark}[theorem]{Remark}
\newtheorem{definition}[theorem]{Definition}

\usepackage{ascmac}
\usepackage{graphicx}
\usepackage{amsmath,amssymb}
\def\inter{\mathrm{Int}}
\title{Hausdorff measures and packing measures of limit sets of CIFSs of generalized complex continued fractions\footnote{Date:\today. Published in J. Difference Equ. Appl. \textbf{26} (2020), no. 1, pp. 104--121. } }
\author{Kanji INUI\footnote{corresponding author}\\ Course of Mathematical Science, Department of Human Coexistence, \\
Graduate School of Human and Environmental Studies, Kyoto University\\
Yoshida-nihonmatsu-cho, Sakyo-ku, Kyoto, 606-8501, JAPAN\\
E-mail: inui.kanji.43a@st.kyoto-u.ac.jp\\
 \\ 
 Hiroki SUMI \\Course of Mathematical Science, Department of Human Coexistence, \\
Graduate School of Human and Environmental Studies, Kyoto University\\
Yoshida-nihonmatsu-cho, Sakyo-ku, Kyoto, 606-8501, JAPAN\\
E-mail: sumi@math.h.kyoto-u.ac.jp\\
Homepage: http://www.math.h.kyoto-u.ac.jp/\textasciitilde sumi/index.html
}

\date{}
\begin{document}
\maketitle
\begin{abstract}
We consider a certain family of CIFSs of the generalized complex continued fractions with a complex parameter space. 
We show that for each CIFS of the family, the Hausdorff measure of the limit set of the CIFS with respect to the Hausdorff dimension is zero and the packing measure of the limit set of the CIFS with respect to the Hausdorff dimension is positive (main result). 
This is a new phenomenon of infinite CIFSs which cannot hold in finite CIFSs. 
We prove the main result by showing some estimates for the unique conformal measure of each CIFS of the family and by using some geometric observations. 
%By the general theories of the finite CIFS, the Hausdorff measure of the limit set of each finite CIFS with respect to the Hausdorff dimension and the packing measure of the limit set with respect to the Hausdorff dimension are positive and finite. 
%However, Theorem \ref{main} indicates that for each $S_{\tau}$ of the family of CIFSs of generalized complex continued fractions, which consists of uncountably many elements, 
\end{abstract}
\section{Introduction}
Recent studies of fractal geometry have been developed in many directions. 
One of the most developed is the study of the limit sets of conformal iterated function systems (for short, CIFSs). 
Indeed, the general theory of limit sets of CIFSs with finitely many mappings (for short, \textit{finite} CIFSs) has been well studied (see \cite{F}, \cite{MU}). 
For example, there exists the formula on the Hausdorff dimension of the limit sets, and there exist statements which claim that the Hausdorff measure of the limit set of any finite CIFS with respect to the Hausdorff dimension is positive and finite and the packing measure of the limit set with respect to the Hausdorff dimension is also positive and finite (from this, we deduce that the Hausdorff dimension of the limit set of any finite CIFS and the packing dimension of the limit set are the same in general). 

On the other hand, studies of limit sets of conformal iterated function systems with infinitely many mappings (for short, \textit{infinite} CIFS) were initiated by Mauldin and Urba\'{n}ski (\cite{MU}, \cite{MU2}, \cite{MU3}) and there are many related results on infinite CIFSs with overlaps by Mihailescu and Urba\'{n}ski (\cite{MiU}, \cite{MiU2}). 
Mauldin and Urba\'{n}ski found a formula on the Hausdorff dimension of limit sets generalizing the above formula on the Hausdorff dimension of limit sets of finite CIFSs. 
In addition, they found a condition under which the Hausdorff measure of the limit set of the infinite CIFS with respect to the Hausdorff dimension is zero. 

Moreover, Mauldin and Urba\'{n}ski constructed an interesting example of an infinite CIFS which is related to the complex continued fractions in the paper \cite{MU}. 
The construction of the example is the following. 
Let $X := \{ z \in \mathbb{C} \ | \ |z-1/2| \leq 1/2 \}$. 
We call $\hat{S} := \{ \hat{\phi}_{(m, n)}(z) \colon X \to X \ | \ (m, n) \in \mathbb{Z} \times \mathbb{N} \}$ the CIFS of complex continued fractions, where $\mathbb{Z}$ is the set of the integers, $\mathbb{N}$ is the set of the positive integers and
\[ \hat{\phi}_{(m, n)}(z) := \frac{1}{z + m + ni} \quad (z \in X). \]
Let $\hat{J}$ be the limit set of $\hat{S}$ (see Definition \ref{CIFSDef}) and $\hat{h}$ be the Hausdorff dimension of $\hat{J}$. 
For each $s \geq 0$, we denote by $\mathcal{H}^{s}$ the $s$-dimensional Hausdorff measure and denoted by $\mathcal{P}^{s}$ the $s$-dimensional packing measure. 
For this example, Mauldin and Urba\'{n}ski showed the following theorem.
\begin{theorem}[D. Mauldin, M. Urbanski (1996)] \label{knownmainresult}
Let $S$ be the CIFS of complex continued fractions. 
Then, we have that $\mathcal{H}^{\hat{h}}(\hat{J}) = 0$ and $\mathcal{P}^{\hat{h}}(\hat{J}) > 0$. 
\end{theorem}
This is an example of infinite CIFS of which the Hausdorff measure of the limit set with respect to the Hausdorff dimension is zero and the packing measure of the limit set is positive. 
Note that this is a new phenomenon of infinite CIFSs which cannot hold in finite CIFSs% since Hausdorff measure of the limit set with respect to the Hausdorff dimension and the packing measure of the limit set with respect to the Hausdorff dimension are positive and finite
. 

It is interesting for us to ask for an infinite CIFS $S$, how often we have the situation that $\mathcal{H}^{h_{S}}(J_{S}) = 0$ and $\mathcal{P}^{h_{S}}(J_{S}) > 0$, where 
we denote by $J_{S}$ the limit set of $S$ (see Definition \ref{CIFSDef}) and $h_{S}$ the Hausdorff dimension of $J_{S}$. 
%Then, there is a natural question. 
%That is, is there another CIFS $S$ such that limit set $J_{S}$ of $S$ satisfies $\mathcal{H}^{h_{S}}(J_{S}) = 0$ and $\mathcal{P}^{h_{S}}(J_{S}) > 0$ ?, where we denote by $h_{S}$ the Hausdorff dimension of $J_{S}$. 
%Unfortunately, it seems difficult to extend Theorem \ref{knownmainresult} in general. 
%It is because Mauldin and Urba\'{n}ski prove Theorem \ref{knownmainresult} by using the Complex analysis techniques and some ``good" results of $S_{0}$. 
%For example, they showed that $S_{0}$ is hereditarily regular (see Definition \ref{RD}), that there exists the conformal measure of $S_{0}$ (see Proposition \ref{existenceofconformalmeasure}) and that $1< h_{0}<2$. 
%In addition,  they showed some ``good" estimate of the conformal measure. 
We considered the generalization of $\hat{S}$ in our previous paper \cite{IOS}. 
That is, we introduced a family of CIFSs of the generalized complex continued fractions $\{ S_{\tau} \}_{\tau \in A_{0}}$ to present new and interesting examples of infinite CIFSs. 
Note that $\{ S_{\tau} \}_{\tau \in A_{0}}$ is a family of CIFSs which has uncountably many elements. 
%In addition, we showed for each $\tau \in A_{0}$, $S_{\tau}$ has some ``good" properties (Lemmas \ref{GCCFisCIFS}, \ref{GCCFisahregularCIFS} $\sim$ \ref{Xinftythorem}). 
The aim of this paper is to generalize Theorem \ref{knownmainresult} and to show that $\mathcal{H}^{h_{S_{\tau}}}(J_{S_{\tau}}) = 0$ and $\mathcal{P}^{h_{S_{\tau}}}(J_{S_{\tau}}) > 0$ for each $\tau \in A_{0}$ to find new and interesting examples of infinite CIFSs with the phenomenon which cannot hold in finite CIFSs. 
%In our previous paper(\cite{IOS}), we introduced the family of the CIFSs related to generalized complex continued fractions. 
%the authors think that the setting of each CIFS of the family is very closed to the setting of the CIFS of complex continued fractions. 
%Therefore, the aim of this paper is to show analogy of Theorem \ref{knownmainresult} for each CIFS of the family. 

The precise statement is the following. 
Let \[A_{0} := \{ \tau = u + iv \in \mathbb{C} \ | \ u \geq 0 \ \mathrm{and} \ v \geq 1 \}\] and $X := \{ z \in \mathbb{C} \ | \ |z-1/2| \leq 1/2 \}$. 
Also, we set $I_{\tau} := \{ m +n\tau \in \mathbb{C} \ |\ m, n \in \mathbb{N} \}$ for each $\tau \in A_{0}$, where $\mathbb{N}$ is the set of the positive integers. 
\begin{definition}[The CIFS of generalized complex continued fractions]
For each $\tau \in A_{0}$, $S_{\tau} := \{ \phi_{b} \colon X \rightarrow X \ |\ b \in I_{\tau} \} $ is called the CIFS of generalized complex continued fractions. Here, 
\[ \phi_{b}(z) := \frac{1}{z + b} \quad ( z \in X ). \]
\end{definition}
The family $\{S_{\tau}\}_{\tau \in A_{0}}$ is called the family of CIFSs of generalized complex continued fractions. 
For each $\tau \in A_{0}$, let $J_{\tau}$ be the limit set of the CIFS $S_{\tau}$ (see Definition \ref{CIFSDef}) and let $h_{\tau}$ be the Hausdorff dimension of the limit set $J_{\tau}$. 

We remark that this family of CIFSs is a generalization of $\hat{S}$ in some sense. 
The system $S_{\tau}$ is related to ``generalized" complex continued fractions since each point of the limit set of $S_{\tau}$ is of the form
\begin{equation*}
\cfrac{1}{b_{1} + \cfrac{1}{b_{2} + \cfrac{1}{b_{3} + \cdots}}}
\end{equation*}
for some sequence $( b_{1}, b_{2}, b_{3}, \ldots) $ in $I_{\tau}$ (see Definition \ref{CIFSDef}). 
Note that there are many kinds of general theories for continued fractions and related iterated function systems (\cite{IK}, \cite{MU}, \cite{MU2}, \cite{MiU2}). 

We now give the main result of this paper.  
\begin{theorem}[Main result] \label{main}
Let $\{S_{\tau}\}_{\tau \in A_{0}}$ be the family of CIFSs of generalized complex continued fractions. Then, for each $\tau \in A_{0}$, we have $\mathcal{H}^{h_{\tau}}(J_{\tau}) = 0$ and $\mathcal{P}^{h_{\tau}}(J_{\tau}) > 0$. 
\end{theorem}
\begin{remark}
It was shown that for each $\tau \in A_{0}$, $\overline{J_{\tau }}\setminus J_{\tau}$ is at most countable and $h_{\tau }=\dim _{\mathcal{H}}(\overline{J_{\tau }}$) (\cite{IOS}). 
Thus, for each $\tau \in A_{0}$, we have $\mathcal{H}^{h_{\tau}}(\overline{J_{\tau }}) = \mathcal{H}^{h_{\tau}}(J_{\tau }) = 0$. 
Also, for each $\tau \in A_{0}$, since the set of attracting fixed points of elements of the semigroup generated by $S_{\tau}$ is dense in $J_{\tau }$, Theorem 1.1 of \cite{St} implies that $\overline{J_{\tau }}$ is equal to the Julia set of the rational semigroup generated by $\{ \phi _{b}^{-1}\mid b\in I_{\tau }\}$.  
\end{remark}
\begin{remark}
By the general theories of finite CIFSs, the Hausdorff measure of the limit set of each finite CIFS with respect to the Hausdorff dimension and the packing measure of the limit set with respect to the Hausdorff dimension is positive and finite. 
However, Theorem \ref{main} indicates that for each $S_{\tau}$ of the family of CIFSs of generalized complex continued fractions, which consists of uncountably many elements, the Hausdorff measure of the limit set with respect to the Hausdorff dimension is zero and the packing measure of the limit set with respect to the Hausdorff dimension is positive. 
This is also a new phenomenon which cannot hold in the finite CIFSs. 
%In addition, there exist the limit sets of infinite CIFSs which happens the new phenomena. 
\end{remark}

Ideas and strategies to prove the main result are the following. 
To prove $\mathcal{H}^{h_{\tau}}(J_{\tau}) = 0$, we use some results from the paper \cite{MU} and some results from our previous paper \cite{IOS}. 
For example, we use the fact that for each $\tau \in A_{0}$, $S_{\tau}$ is hereditarily regular (see Lemma \ref{GCCFisahregularCIFS}), that for each $\tau \in A_{0}$, there exists a conformal measure of $S_{\tau}$ (see Proposition \ref{existenceofconformalmeasure}) and that for each $\tau \in A_{0}$, $1 < h_{\tau} < 2$ (see Lemma \ref{betweenoneandtwo}). 

In order to prove the main results, we show that for each $\tau \in A_{0}$, $S_{\tau}$ satisfies the assumption of Theorem \ref{Hausdorffmeasureestimate}. 
That is, by using Lemma \ref{Xinftythorem}, we need to show that there exists a sequence $\{r_{j}\}_{j=1}^{\infty}$ in the set of positive real numbers such that \begin{equation} \label{keyassumption}
\limsup_{j \to \infty} \frac{m_{\tau}(B(0, r_{j}))}{r_{j}^{h_{\tau}}} = \infty, 
\end{equation} where $m_{\tau}$ is the conformal measure of $S_{\tau}$ (see Proposition \ref{existenceofconformalmeasure}). 

It is worth pointing out that in order to prove the main result, we use some sharp estimates on the values of the conformal measure. 
In order to prove the sharp estimates, we have to show another basic estimate at first (see Lemma \ref{basicestimate}). 
By the properties of the conformal measure (see Proposition \ref{existenceofconformalmeasure}) and by Lemma \ref{basicestimate}, we have that there exists $K_{0} \geq 1$ such that for each $b \in I_{\tau}$, $\phi_{b}(X) \subset B(0, K_{0} |b|^{-1})$ and 
\begin{equation} \label{conformalBDPestimate}
m_{\tau}(\phi_{b}(X)) \geq \int_{X} |\phi_{b}^{\prime}|^{h_{\tau}} \mathrm{d} m_{\tau} \geq (K_{0}^{-1}|b|^{-2})^{h_{\tau}}m_{\tau}(X) \geq K_{0}^{-h_{\tau}}|b|^{-2h_{\tau}}. 
\end{equation}
Moreover, by using properties of the conformal measure (see Proposition \ref{existenceofconformalmeasure}), we have that for all $b, b^{\prime} \in I_{\tau}$ with $b \neq b^{\prime}$, $m_{\tau}(\phi_{b}(X) \cap \phi_{b^{\prime}}(X) ) = 0$. 
Then, for each $\tau \in A_{0}$, let $N_{\tau} \in \mathbb{N}$ be large enough and for each $r>0$, we set $I_{\tau}(r) := \{ b \in I_{\tau} \ | \ K_{0}r^{-1} < |b| \leq N_{\tau}K_{0}r^{-1} \}$. 
Note that if $b \in I_{\tau}(r)$, then $|b|^{-1} > K_{0}^{-1} N_{\tau}^{-1} r$ and $m_{\tau}(\phi_{b}(X)) \gtrsim r^{2h_{\tau}}$. 

We next show some basic results on the estimate of $|I_{\tau}(r)|$ (see Lemma \ref{insidetwocircle}, Proposition \ref{latticepointlemma} and Lemma \ref{latticepointestimate}) by the general theory of linear algebra and elementary geometric observations. 
By these results, we show that if $r >0$ is small enough, then $|I_{\tau}(r)| \gtrsim r^{-2}$ (see inequality (\ref{taulatticepointestimate})). 
In addition, by the inequality (\ref{conformalBDPestimate}), we deduce that for $b \in I_{\tau}(r)$, we have $B(0, r) \supset B(0, K_{0}|b|^{-1}) \supset \phi_{b}(X)$ and if $r>0$ is small enough, then 
\begin{equation} \label{conformalroughestimate}
m_{\tau}(B(0, r)) \geq \sum_{b \in I_{\tau}(r)} m_{\tau}(\phi_{b}(X)) \gtrsim \sum_{b \in I_{\tau}(r)} |b|^{-2h_{\tau}} \gtrsim \sum_{b \in I_{\tau}(r)} r^{2h_{\tau}} = |I_{\tau}(r)| \ r^{2h_{\tau}} \gtrsim r^{2h_{\tau}-2}
\end{equation}
(see inequality (\ref{conformalmeasureestimate})). 
By the inequality (\ref{conformalroughestimate}), we finally show that if $r > 0$ is small enough, 
\begin{equation*}
\frac{m_{\tau}(B(0, r))}{r^{h_{\tau}}} \gtrsim r^{h_{\tau}-2} = \left( \frac{1}{r} \right)^{2-h_{\tau}}. 
\end{equation*}
By Lemma \ref{betweenoneandtwo} (that is, $2 > h_{\tau}$), we deduce (\ref{keyassumption}). 

To prove $\mathcal{P}^{h_{\tau}}(J_{\tau}) > 0$, we need to show for each $\tau \in A_{0}$, $S_{\tau}$ satisfies the assumption of Theorem \ref{packingmeasureestimate}. 
That is, we need to show that for each $\tau \in A_{0}$, $J_{\tau} \cap \inter(X) \neq \emptyset$. 
By geometric observations and some properties of $\phi_{b} \in S_{\tau}$, we obtain that $J_{S} \cap \inter(X) \neq \emptyset$. 

The rest of the paper is organized as follows. 
In Section 2, we summarize the theory of CIFSs without proofs. 
In Section 3, we give the proofs of some properties of the CIFS of the generalized complex continued fractions. 
%These results are already proved in the paper \cite{IOS}. 
In Section 4, we prove the main result of this paper. \\
\\
\textbf{Acknowledgement. }  
The authors thank Rich Stankewitz for valuable comments.
The second author is partially supported by JSPS Kakenhi JP 18H03671, JP 19H01790.
%
%
%
%
%
%
%\section{Preliminary}
%
\section{Conformal iterated function systems}
In this section, we recall general settings of CIFSs (\cite{MU}, \cite{MU2}). 
\begin{definition}[Conformal iterated function system] \label{CIFSDef}
Let $X \subset \mathbb{R}^{d}$ be a non-empty compact and connected set and let $I$ be a finite set or bijective to $\mathbb{N}$. Suppose that $I$ has at least two elements. 
We say that $S := \{ \phi_{i} \colon X \to X \ |\ i \in I \}$ is a conformal iterated function system (for short, CIFS) if $S$ satisfies the following conditions. 
\begin{enumerate}
\item Injectivity: For all $i \in I$, $\phi_{i} \colon X \to X$ is injective. 
\item Uniform Contractivity: There exists $c \in (0, 1) $ such that, for all $i \in I$ and $x, y \in X $, the following inequality holds. 
	\[ | \phi_{i}(x) - \phi_{i}(y) | \leq c| x - y |. \] 
\item Cone Condition: For all $x \in \partial X$, there exists an open cone $\mathrm{Con}(x, u, \alpha)$ with a vertex $x$, a direction $u$, an altitude $|u|$ and an angle $\alpha$ such that $\mathrm{Con}(x, u, \alpha)$ is a subset of $\inter(X)$. 
\item Open Set Condition(OSC): For all $ i, j \in I \ ( i \neq j )$, $ \phi_{i}(\inter(X)) \subset \inter(X)$ and  $ \phi_{i}(\inter(X)) \cap  \phi_{j}(\inter(X)) = \emptyset $. 
Here, $\inter(X)$ denotes the set of interior points of $X$ with respect to the topology in $\mathbb{R}^{d}$. 
\item Bounded Distortion Property(BDP): There exists $K \geq 1 $ and an open and connected subset $V \subset \mathbb{R}^{d}$ with $X \subset V$ such that for all $x, y \in V $ and for all $w \in I^{*} := \bigcup_{n=1}^{\infty} I^{n}$, the following inequality holds. \label{openconnectedV}
\[ |\phi^{\prime}_{w}(x)| \leq K \cdot |\phi^{\prime}_{w}(y)|. \]
Here, for each $n \in \mathbb{N}$ and $w = w_{1}w_{2} \cdots w_{n} \in I^{n}$, we set $\phi_{w} := \phi_{w_{1}} \circ \phi_{w_{2}} \circ \cdots \circ \phi_{w_{n}}$ and $\displaystyle |\phi^{\prime}_{w}(x)|$ denotes the  norm of the derivative of $\phi_{w}$ at $x \in X$ with respect to the Euclidean metric on $\mathbb{R}^{d}$. 
\item Conformality: There exists a positive number $\epsilon$ such that for all $i \in I $, $\phi_{i} $ extends to a $C^{1+\epsilon}$diffeomorphism on $V$ and $\phi_{i} $ is conformal on $V$, where $V$ is the open and connected subset introduced in \ref{openconnectedV}.  
\end{enumerate}
%$I$ is called an alphabet. 
We set $I^{*} := \bigcup_{n = 1}^{\infty} I^{n}$. 
We endow $I$ with the discrete topology, and endow $I^{\infty} := I^{\mathbb{N}}$ with the product topology. 
Note that $I^{\infty}$ is a Polish space. In addition, if $I$ is a finite set, then $I^{\infty}$ is a compact metrizable space. 

Let $S$ be a CIFS. 
For each $w=w_{1}w_{2}w_{3} \cdots \in I^{\infty}$, we set $w|_{n} := w_{1}w_{2} \cdots w_{n} \in I^{n}$ and $\phi_{w|_{n}} := \phi_{w_{1}} \circ \phi_{w_{2}} \circ \cdots \circ \phi_{w_{n}}$. 
Then, we have $\bigcap_{n \in \mathbb{N}} \phi_{w|_{n}}(X)$ is a singleton. 
We denote it by $\{ x_{w} \}$. 
The coding map $\pi \colon I^{\infty} \rightarrow X$ of $S$ is defined by $w \mapsto x_{w}$. 
Note that $\pi \colon I^{\infty} \rightarrow X$ is continuous. 
A limit set of $S$ is defined by 
\[ J_{S} :=\pi(I^{\infty}) = \bigcup_{w \in I^{\infty}} \bigcap_{n \in \mathbb{N}} \phi_{w|_{n}}(X).  \]
\end{definition}
Note that for all CIFS $S$, $J_{S}$ is Borel set in X. 
For each IFS $S$, we set $h_{S} := \dim_{\mathcal{H}}J_{S}$, where $\dim_{\mathcal{H}}$ denotes the Hausdorff dimension. 

For any CIFS $S$, we define the pressure function of $S$ as follows. 
\begin{definition}[Pressure function]\label{PSD}
For each $n\in \mathbb{N}$, $[0, \infty]$-valued function $\psi^{n}_{S}$ is defined by 
\[\psi^{n}_{S}(t) := \sum_{w \in I^{n}} \left( \sup_{z \in X}|\phi_{w}^{\prime}(z)| \right)^{t} \quad (t \geq 0). \]

We set \[\displaystyle P_{S}(t) := \lim_{n \to \infty} \frac{1}{n} \log \psi^{n}_{S}(t) \in (-\infty, \infty]. \]
The function $P_{S} \colon [0, \infty) \to (-\infty, \infty]$ is called the pressure function of $S$. 
\end{definition}
Note that for all $t \geq 0$, $P_{S}(t)$ exists because of the following proposition. 
\begin{proposition}[\cite{MU} Lemma 3.2]\label{logsubadditive}
For all $m, n \in \mathbb{N}$ and $t \geq 0$, we have $\psi_{S}^{m+n}(t) \leq \psi_{S}^{m}(t)\psi_{S}^{n}(t)$. 
In particular, for all $t \geq 0$, $\log\psi_{S}^{n}(t)$ is subadditive with respect to $n \in \mathbb{N}$. 
\end{proposition}
We set $\theta_{S} := \inf\{ t \geq 0 \ |\ \psi_{S}^{1}(t) < \infty \}$. 
By using the pressure function, we define properties of CIFSs. 
\begin{definition}[Regular, Strongly regular, Hereditarily regular]\label{RD}
Let $S$ be a CIFS. 
We say that $S$ is regular if there exists $t \geq 0$ such that $P_{S}(t) = 0$. 
We say that $S$ is strongly regular if there exists $t \geq 0$ such that $P_{S}(t) \in (0, \infty)$. 
We say that $S$ is hereditarily regular if for all $I^{\prime} \subset I$ with $|I \setminus I^{\prime}| < \infty $, $S^{\prime} := \{ \phi_{i} \colon X \to X \ |\ i \in I^{\prime} \} $ is regular. 
Here, for any set $A$, we denote by $|A|$ the cardinality of $A$. 
\end{definition}
Note that if a CIFS $S$ is hereditarily regular, then $S$ is strong regular and if $S$ is strong regular, then $S$ is regular. 

%We introduce some additional notations.  
\begin{definition} \label{Xinftydefinition}
Let $S$ be a CIFS. We write $S$ as $\{\phi_{i}\}_{i \in I}$. 
Suppose that $I$ is a countable infinite set. 
Let $z \in X$ and $\{ z_{i} \}_{i \in I^{\prime}} \subset X$ with $I^{\prime} \subset I$ and $|I^{\prime}|=\infty$. 
We say that $\lim_{i \in I^{\prime}}z_{i} = z$ if for each $\epsilon > 0$, there exists $F^{\prime} \subset I^{\prime}$ with $|F^{\prime}| < \infty$ such that if $i \in I^{\prime} \setminus F^{\prime}$, then $|z_{i} - z| < \epsilon $. 
We set
\[ X_{S}(\infty) := \{ \lim_{i \in I^{\prime}}z_{i} \in X \ | \ ^{\exists}I^{\prime} \subset I, ^{\exists}\{z_{i}\}_{i \in I^{\prime}} \ \text{s.t.} \ |I^{\prime}|=\infty, z_{i} \in \phi_{i}(X) \ \ (i \in I^{\prime})\}. \]
Equivalently, $X_{S}(\infty)$ is the set of accumulation points of sequences in $\phi_{i}(X)$, $i \in I$, i.e. limits of infinite sequences from $\phi_{i}(X)$, $i \in I$.  
\end{definition}
%\begin{proposition}
%Let S be a CIFS. If S is regular, there exists the unique probability measure $m_{S}$ on $X$ such that following properties holds. 
%\begin{enumerate}
%\item $m_{S}(J_{S}) = 1$. 
%\item For all borel subset $A$ on $X$ and $i \in I$, $m_{S}(\phi_{i}(A)) = \int_{A}|\phi_{i}^{\prime}|^{h_{S}}  \mathrm{d} m_{S}$
%\item For all $i, j \in I$ with $i \neq j$, $m_{S}(\phi_{i}(X) \cap \phi_{j}(X) ) = 0$
%\end{enumerate}
%\end{proposition}
%This probability measure $m_{S}$ is called $h_{S}$-conformal measure. 
%
%
%
%
%
%We set $F(I) := \{ F \subset I \ |\ |F| < \infty \}$. For each $F \in F(I)$, we set $S_{F} := \{ \phi_{i} \colon X \to X |\ i \in F \}$. 
Mauldin and Urba\'{n}ski showed the following results. 
Recall that $h_{S} := \dim_{\mathcal{H}} J_{S}$, where $\dim_{\mathcal{H}} J_{S}$ denotes the Hausdorff dimension of the limit set of $S$. 
%\begin{theorem}[\cite{MU} Theorem 3.15]\label{presshaus}
%Let $S$ be a CIFS. Then we have 
%\[h_{S} = \inf\{ t \geq 0 \ |\ P_{S}(t) < 0 \} \geq \theta_{S}.\]
%Moreover, if there exists $t \geq 0$ such that $P_{S}(t) = 0$, then $t$ is the unique zero of the pressure function $P_{S}$ and we have $t = h_{S}$. 
%\end{theorem}
%
\begin{theorem}[\cite{MU} Theorem 3.20]\label{hregularequivalent}
Let $I$ be infinite and let $S$ be a CIFS. Then, the following conditions are equivalent.  
\begin{enumerate}
\item $S$ is hereditarily regular. 
\item $\psi^{1}_{S}(\theta_{S}) = \infty$. 
\end{enumerate}
Especially, if $S$ is hereditarily regular, then we have $ \theta_{S} < h_{S} $.
\end{theorem}
%
%\begin{theorem}[\cite{MU} Proposition 4.4] \label{Hdimupperestimate}
%Let $S$ be a regular CIFS. Then if $\lambda_{d}(\inter(X) \setminus X_{1}) > 0$, then $h_{S} < d$. Here, $\lambda_{d}$ is the $d$-dimensional Lebesgue measure and $X_{1} := \cup_{i \in I} \phi_{i}(X)$. 
%\end{theorem}
%
\begin{proposition}[\cite{MU} Lemma 3.13] \label{existenceofconformalmeasure}
Let $S$ be a CIFS. 
If $S$ is regular, then there exists the unique Borel probability measure $m_{S}$ on $X$ such that the following properties hold. 
\begin{enumerate}
\item $m_{S}(J_{S}) = 1$. 
\item For all Borel subset $A$ on $X$ and $i \in I$, $m_{S}(\phi_{i}(A)) = \int_{A}|\phi_{i}^{\prime}|^{h_{S}}  \mathrm{d} m_{S}$. 
\item For all $i, j \in I$ with $i \neq j$, $m_{S}(\phi_{i}(X) \cap \phi_{j}(X) ) = 0$. 
\end{enumerate}
\end{proposition}
We call $m_{S}$  the  $h_{S}$-conformal measure of $S$. 
\begin{theorem}[\cite{MU} Theorem 4.9] \label{Hausdorffmeasureestimate}
Let $S$ be a regular CIFS and $m_{S}$ be the $h_{S}$-conformal measure of $S$. 
We set $r_{0} := \mathrm{dist}(X, \partial V)$. 
If there exist a sequence of $\{z_{j}\}_{j=1}^{\infty}$ in $ X_{S}(\infty) $ and a sequence $\{r_{j}\}_{j=1}^{\infty}$ in $(0, r_{0})$ such that 
\[ \limsup_{j \to \infty} \frac{m_{S}(B(z_{j}, r_{j}))}{r_{j}^{h_{S}}} = \infty, \]
then we have $\mathcal{H}^{h_{S}}(J_{S}) = 0$. 
\end{theorem}

%In order to , we introduce a important condition. 
%\begin{definition}[Pointwise finite]
%Let $S$ be a IFS. 
%We say that $S$ is pointwise finite if for all $x \in X$, $\{i \in I | x \in \phi_{i}(X)\}$ is a finite set. 
%\end{definition}
%Note that if $I$ is a finite set, $S$ is pointwize finite. 

\begin{theorem}[\cite{MU} Lemma 4.3] \label{packingmeasureestimate}
Let $S$ be a regular CIFS. 
If $J_{S} \cap \inter(X) \neq \emptyset$, then we have $\mathcal{P}^{h_{S}}(J_{S}) > 0$. 
\end{theorem}
\section{CIFSs of generalized complex continued fractions}
In this section, we prove some properties of the CIFSs of generalized complex continued fractions (\cite{IOS}). 
Note that they are important and interesting examples of infinite CIFSs. 
We introduce some additional notations. 
For each $\tau \in A_{0}$, we set $\pi_{\tau} := \pi_{S_{\tau}}$, $\theta_{\tau} := \theta_{S_{\tau}}$, $\psi_{\tau}^{n}(t):= \psi_{S_{\tau}}^{n}(t) \ (t \geq 0, n \in \mathbb{N})$, $P_{\tau}(t) := P_{S_{\tau}}(t) $ \ $(t \geq 0)$ and $X_{\tau}(\infty) := X_{S_{\tau}}(\infty)$. 

In order to prove the main result, we need the following lemmas \ref{GCCFisCIFS} $\sim$ \ref{Xinftythorem} which were shown in \cite{IOS}. 
For the readers, we give the proofs. 
\begin{lemma}\label{GCCFisCIFS}
For all $\tau \in A_{0}$, $S_{\tau}$ is a CIFS. 
\end{lemma}
\begin{proof}
Let $\tau \in A_{0}$. Firstly, we show that for all $b \in I_{\tau}$, $\phi_{b}(X) \subset X$. %$\hat{C}$
Let $Y := \{ z \in \mathbb{C} |\ \Re z \geq 1 \}$ and let $f \colon \hat{\mathbb{C}} \to \hat{\mathbb{C}} $ be the M\"{o}bius transformation defined by $f(z) := 1/z$. 
Since $f(0) = \infty$, $f(1) = 1$, $f(1/2+i/2) = 2/(1+i) = (1-i)$, we have $f(\partial X) = \partial Y \cup \{ \infty \} $.
Moreover, since f(1/2) = 2, we have $f(X) = Y \cup \{ \infty \}$. 
Thus, $f \colon X \to Y \cup \{ \infty \}$ is a homeomorphism. 
Let $g_{b} \colon X \to Y$ be the map defined by $g_{b}(z) := z+b$. 
We deduce that $\phi_{b} = f^{-1} \circ g_{b}$ and $\phi_{b}(X) \subset f^{-1}(Y) \subset X$. 
Therefore, we have proved $\phi_{b}(X) \subset X$. 

We next show that for each $\tau \in A_{0}$, $S_{\tau}$ satisfies the conditions of Definition \ref{CIFSDef}. 

1. Injectivity. \\
Since each $\phi_{b}$ is a M\"{o}bius transformation, each $\phi_{b}$ is injective. 

2. Uniform Contractivity. \\
Let $b=m+n\tau (=m+nu+inv)$ be an element of $I_{\tau}$ and let $z=x+iy$ and $z^{\prime}=x^{\prime}+iy^{\prime}$ be elements of $X$. 
We have

\begin{align*}
|z+b|^{2}	&= |x+m+nu+i(y+nv)|^{2}\\
			&= (x+m+nu)^{2}+(y+nv)^{2} \geq (0+1+0)^{2} + (-1/2+1)^{2} = \frac{5}{4}. 
\end{align*}
Therefore, we deduce that $|z+b| \geq \sqrt{5/4} $. 
We also deduce that $|z^{\prime}+b| \geq \sqrt{5/4} $. 
Finally, we obtain that

\begin{align*}
|\phi_{b}(z)-\phi_{b}(z^{\prime})|	&= \left| \displaystyle \frac{1}{z+b} - \frac{1}{z^{\prime}+b} \right|\\
									&= \frac{|z-z^{\prime}|}{|z+b||z^{\prime}+b|} \leq \left( \sqrt{ \frac{4}{5} } \right)^{2} |z-z^{\prime}| = \frac{4}{5}|z-z^{\prime}|.
\end{align*}
Therefore, $S_{\tau}$ is uniformly contractive on $X$. 

3. Cone Condition. \\
Since $X$ is a closed disk, the Cone Condition is satisfied.

4. Open Set Condition. \\
Note that $\inter(X) = \{ z \in \mathbb{C} |\ | z - 1/2| < 1/2 \}$. 
Let $\tau \in A_{0}$ and let $b \in I_{\tau}$. 
Since $f(\partial X) = \partial Y \cup \{ \infty \} $, we deduce that for all $b \in I_{\tau}$, 
\[ g_{b}(\inter(X)) \subset \{ z=x+iy \in \mathbb{C} |\ x > 1 \} = f(\inter(X)). \]
Moreover, if $b$ and $b^{\prime}$ are distinct elements, then $g_{b}(\inter(X))$ and $g_{b^{\prime}}(\inter(X))$ are disjoint. 
Therefore, we have that for all $b \in I_{\tau}$, 
\[\phi_{b}(\inter(X)) = f^{-1}\circ g_{b}(\inter(X)) \subset f^{-1}\circ f(\inter(X)) = \inter(X). \]
And if $b$ and $b^{\prime}$ are distinct elements, 
\[\phi_{b}(\inter(X)) \cap \phi_{b^{\prime}}(\inter(X))	= f^{-1}(g_{b}(\inter(X)) \cap g_{b^{\prime}}(\inter(X))) = \emptyset. \]
Therefore, $S_{\tau}$ satisfies the Open Set Condition of $S_{\tau}$. 

5. Bounded Distortion Property. \\
Let $\epsilon$ be a positive real number which is less than $1/12$ and let $V^{\prime} := B(1/2, 1/2+\epsilon)$ be the open ball with center $1/2$ and radius $1/2+\epsilon$. We set $\tau := u +iv$. 
Then, for all $(m, n) \in \mathbb{N}^{2}$ and $z := x+iy \in V^{\prime}$, 
we have that 

\begin{align*}
|\phi_{m+n\tau}^{\prime}(z)|	&= \frac{1}{|z+m+n\tau|^{2}} = \frac{1}{(x+m+nu)^{2}+(y+nv)^{2}}\\
                                &\leq \frac{1}{(-\epsilon+1+0)^{2}+(-1/2-\epsilon+1)^{2}}\\
                                &= \frac{1}{2\epsilon^{2}-3\epsilon+5/4} = \frac{1}{2(\epsilon-3/4)^{2}+1/8}\\
                                &\leq \frac{1}{2(1/12-3/4)^{2}+1/8} = \frac{72}{73} < 1
\end{align*}
For each $z \in V^{\prime}$, we set 

\[z^{\prime} :=	\begin{cases}
					\displaystyle (|z-1/2|-\epsilon)\frac{(z-1/2)}{|z-1/2|} +1/2 & (z \notin X) \\
                    z & (z \in X).
                    \end{cases}
\]
Then, we have that $|z-z^{\prime}| \leq \epsilon$ and $|z^{\prime}-1/2| < 1/2$. 
It implies that $z^{\prime} \in X $. 
Thus, we obtain that $|\phi_{b}(z) - \phi_{b}(z^{\prime})| \leq (72/73) |z-z^{\prime}| < \epsilon$ and  
\[\left| \phi_{b}(z) - \frac{1}{2} \right|	\leq |\phi_{b}(z) - \phi_{b}(z^{\prime})| + \left| \phi_{b}(z^{\prime}) - \frac{1}{2} \right| < \frac{1}{2} + \epsilon. \]
It follows that for all $b \in I_{\tau}$, $\phi_{b}(V^{\prime}) \subset V^{\prime}$. 
In addition, $\phi_{b}$ is injective on $V^{\prime}$ and $\phi_{b}$ is holomorphic on $V^{\prime} := B(1/2, 1/2+\epsilon)$ since $\phi_{b}$ is holomorphic on $\mathbb{C} \setminus \{-b\}$. 

Let $b$ be an element of $I_{\tau}$ and $r_{0}:= 1/2 + \epsilon$.
Let $f_{b}$ be the function defined by 
\[f_{b}(z) := \frac{(\phi_{b}(r_{0}z+1/2) - \phi_{b}(1/2))}{r_{0}\phi_{b}^{\prime}(1/2)} \quad (z \in D:=\{z \in \mathbb{C} | |z| < 1\}). \]
Note that $f_{b}$ is holomorphic on $D$ and $f_{b}(0) = 0$ and $f_{b}^{\prime}(0) = 1$. 
By using the Koebe distortion theorem (For example, see \cite[Theorem 4.1.1]{MU3}), we deduce that for all $z \in D$, 
\[\frac{1-|z|}{(1+|z|)^{3}} \leq |f_{b}(z)| \leq \frac{1+|z|}{(1-|z|)^{3}}. \]
Let $r_{1} := (r_{0} +1/2)/2$. we deduce that there exist $C_{1} \geq 1$ and $C_{2} \leq 1$ such that for all $z \in B(0, r_{1}/r_{0}) (\subset D)$, 
\[C_{2} \leq \frac{1-|z|}{(1+|z|)^{3}} \quad \text{and} \quad \frac{1+|z|}{(1-|z|)^{3}} \leq C_{1}. \]
Let $C := C_{1}/C_{2}$. Then, we have that for all $z, z^{\prime} \in B(0, r_{1}/r_{0})$

\begin{align*}
\frac{|\phi_{b}^{\prime}(r_{0}z+1/2)|}{|\phi_{b}^{\prime}(1/2)|}	&= |f_{b}^{\prime}(z)| \leq \frac{1+|z|}{(1-|z|)^{3}} \\
                                                                    &\leq C_{1} = C C_{2} \leq C \frac{1-|z^{\prime}|}{(1+|z^{\prime}|)^{3}}\\
                                                                    &\leq C |f_{b}^{\prime}(z^{\prime})| \leq C \frac{|\phi_{b}^{\prime}(r_{0}z^{\prime}+1/2)|}{|\phi_{b}^{\prime}(1/2)|}. 
\end{align*}
It follows that for all $z, z^{\prime} \in B(0, r_{1}/r_{0})$, $|\phi_{b}^{\prime}(r_{0}z+1/2)| \leq C |\phi_{b}^{\prime}(r_{0}z^{\prime}+1/2)|$. 
Finally, let $V := B(1/2, r_{1})$ be the open ball with center $1/2$ and radius $r_{1}$. 
Then, $V$ is an open and connected subset of $\mathbb{C}$ with $X \subset V$
and for all $z, z^{\prime} \in V$, 
\[|\phi_{b}^{\prime}(z)| \leq C |\phi_{b}^{\prime}(z^{\prime})|. \]
Therefore, $S_{\tau}$ satisfies the Bounded Distortion Property. 

6. Conformality. \\
Let $\tau \in A_{0}$ and let $b \in I_{\tau}$. 
Since $\phi_{b}$ is holomorphic on $\mathbb{C} \setminus \{-b\}$, $\phi_{b}$ is $\mathrm{C}^{2}$ and conformal on $V$. 
By the above argument, we have $\phi_{b}(V) \subset V$. 
\end{proof}

For the rest of the paper, let $V := B(1/2, r_{1})$, where $r_{1}$ is the number in the proof of Lemma \ref{GCCFisCIFS}.  

\begin{lemma}[basic inequality]\label{basicestimate}
Let $\tau \in A_{0}$. 
Then, there exists $K_{0} \geq 1$ such that for all $K \geq K_{0}$ and $b \in I_{\tau}$, the following properties hold. 
\begin{enumerate}
\item $\phi_{b}(V) \subset B(0, K |b|^{-1})$. 
\item For each $z \in V$, $K^{-1} |b|^{-2} \leq |\phi_{b}^{\prime}(z)| \leq K |b|^{-2}$. 
\end{enumerate}
\end{lemma}
\begin{proof}
We use the notations in the proof of Lemma \ref{GCCFisCIFS}. 
Note that $r_{1} \in (1/2, 13/24)$. 
Let $\tau \in A_{0}$ and $b \in I_{\tau}$. 
Since there exists $M \in \mathbb{N}$ such that for all $z \in V = B(1/2, r_{1})$ and $b \in I_{\tau}$, we have that $|b| \leq M|b+z|$ , we deduce that 
\begin{equation} \label{basicinequality1}
|\phi_{b}(z)| \leq M|b|^{-1}. 
\end{equation}
Note that by using the BDP, there exists $C \geq 1$ such that for each $z \in V$, we have 
\begin{equation} \label{basicinequality2}
C^{-1}|\phi_{b}^{\prime}(0)| \leq |\phi_{b}^{\prime}(z)| \leq C |\phi_{b}^{\prime}(0)|. 
\end{equation}
We set $K_{0} := \max \{ C, M \} ( \geq 1)$. 
Let $K \geq K_{0}$. 

By the inequality (\ref{basicinequality1}), we deduce that $\phi_{b}(V) \subset B(0, K |b|^{-1})$. 
By the inequality (\ref{basicinequality2}) and the equality $|\phi_{b}^{\prime}(0)| = |b|^{-2}$, we deduce that for each $z \in V$, $K^{-1}|b|^{-2} \leq |\phi_{b}^{\prime}(z)| \leq K |b|^{-2}$. 
Therefore, we have proved our lemma. 
\end{proof}
\begin{lemma}\label{GCCFisahregularCIFS}
For all $\tau \in A_{0}$, $S_{\tau}$ is a hereditarily regular CIFS with $\theta_{\tau} = 1$. 
\end{lemma}
\begin{proof}
Let $\tau \in A_{0}$. 
For each non-negative integer $p$, we define $K^{\prime}(p):= \{ b = m + n \tau \in I_{\tau} \ |\ (m, n) \in \mathbb{N}^{2}, m < 2^{p}, n < 2^{p} \}$ and $K(p) := K^{\prime}(p) \setminus K^{\prime}(p-1)$. 
Note that for each non-negative integer $p$, $|K^{\prime}(p)|= (2^{p}-1)^{2}$. 
We deduce that for each $p \in \mathbb{N}$, $|K(p)|=|K^{\prime}(p)|-|K^{\prime}(p-1)|=(2^{p}-1)^{2}-(2^{p-1}-1)^{2}= 3 \cdot 4^{p-1}-2 \cdot 2^{p-1} =2^{p-1}(3\cdot2^{p-1}-2)$ and $ 4^{p-1}\leq |K(p)| \leq 3\cdot 4^{p-1}$. 

Let $ b = m + n \tau = m + n (u + i v) \in K(p)$. We consider the following two cases. 
\begin{itemize}
\item[(i)] If $m \geq 2^{p-1}$ then we have 
\begin{align*}
|b|^{2} 
&= |m + nu + inv|^{2} 
= (m + nu)^{2} + (nv)^{2}\\
&\geq (2^{p-1} + u)^{2} + v^{2} 
\geq (2^{p-1})^{2} + |\tau|^{2} 
= 4^{p-1}\left( 1 + \frac{|\tau|^{2}}{4^{p-1}} \right). 
\end{align*}
\item[(ii)] If $n \geq 2^{p-1}$ then we have
\begin{align*}
|b|^{2} &= |m + nu + inv|^{2} 
= (m + nu)^{2} + (nv)^{2}
\geq n^{2}(u^{2} + v^{2}) \geq 4^{p-1}|\tau|^{2}. 
\end{align*}
\end{itemize}
Then for any $t \geq 0$, we have 

\begin{align*}
\sum_{b \in I_{\tau}} |b|^{-2t} 
= \sum_{p \in \mathbb{N}} \sum_{b \in K(p)} \left\{ |b|^{2} \right\}^{-t} 
&\leq \sum_{p \in \mathbb{N}} |K(p)|  4^{-t(p-1)} \left\{ \min \{ 1 + \frac{|\tau|^{2}}{4^{p-1}}, |\tau|^{2} \} \right\}^{-t} \\
&\leq \sum_{p \in \mathbb{N}}  3 \cdot 4^{ (p-1)(1-t)} \left\{ \min \{ 1 + \frac{|\tau|^{2}}{4^{p-1}}, |\tau|^{2} \} \right\}^{-t}. 
\end{align*}
Hence, we deduce that 
\begin{equation}\label{upperestimate}
\sum_{b \in I_{\tau}} |b|^{-2t} \leq 3 \sum_{p \in \mathbb{N}}  4^{(p-1)(1-t)} \left\{ \min \{ 1 + \frac{|\tau|^{2}}{4^{p-1}}, |\tau|^{2} \} \right\}^{-t}. 
\end{equation}
Moreover, by the inequality $|\tau|^{2} \geq 1$ and the inequality $\displaystyle 1 + \frac{|\tau|^{2}}{4^{p-1}} \geq 1$, we deduce that for all $p \in \mathbb{N}$,  
\begin{equation}\label{dominated}
3 \cdot 4^{(p-1)(1-t)} \left\{ \min \{ 1 + \frac{|\tau|^{2}}{4^{p-1}}, |\tau|^{2} \} \right\}^{-t} \leq 3 \cdot 4^{(p-1)(1-t)}. 
\end{equation}
Also, by the inequality $|b| \leq |m| + |n||\tau| \leq 2^{p}(1 + |\tau|) $, we have
\begin{equation*}
\sum_{b \in I_{\tau}} |b|^{-2t} 
= \sum_{p \in \mathbb{N}} \sum_{b \in K(p)} \left\{ |b|^{-2} \right\}^{t} 
\geq \sum_{p \in \mathbb{N}} |K(p)| 4^{-pt}(1 + |\tau|)^{-2t}. 
\end{equation*}
Thus, we deduce that 
\begin{equation}\label{lowerestimate}
\sum_{b \in I_{\tau}} |b|^{-2t} \geq 4^{-1} \sum_{p \in \mathbb{N}} 4^{p(1-t)}(1 + |\tau|)^{-2t}. 
\end{equation}
Finally, from Lemma \ref{basicestimate}, the inequality (\ref{upperestimate}) and the inequality (\ref{lowerestimate}), it follows that $\psi_{\tau}^{1}(1) = \infty$ and if $t >1$, then $\psi_{\tau}^{1}(t) < \infty$. 
Therefore, we deduce that $\theta_{\tau} = 1$ and by Theorem \ref{hregularequivalent}, we obtain that for all $\tau \in A_{0}$, $S_{\tau}$ is hereditarily regular. Hence, we have proved our lemma.  
\end{proof}
\begin{lemma}\label{betweenoneandtwo}
Let $\tau \in A_{0}$. Then we have $1 < h_{\tau} < 2$. 
\end{lemma}
\begin{proof}
Let $\tau \in A_{0}$. 
By Theorem \ref{hregularequivalent} and Lemma \ref{GCCFisahregularCIFS}, we have $ 1 = \theta_{\tau} < h_{\tau} $. 
We now show that $h_{\tau} < 2$. 
We use the notations in the proof of Proposition \ref{GCCFisCIFS}. 
We have
\[\bigcup_{b \in I_{\tau}}g_{b}(X) \subset \{ z \in \mathbb{C} \ |\ \Re z \geq 1 \ \text{and} \ \Im z \geq 0 \}. \]
Let $U_{0}$ be an open ball such that $U_{0} \subset \{ z \in \mathbb{C} \ |\ \Re z > 1 \ \text{and} \ \Im z < 0 \}$. 
Since $U_{0} \subset Y$, we deduce that $f^{-1}(U_{0}) \subset f^{-1}(Y) = \inter(X)$.
We set $X_{1} := \cup_{b \in I_{\tau}} \phi_{b}(X)$. 
Since $U_{0} \cap \bigcup_{b \in I_{\tau}}g_{b}(X) = \emptyset$, we deduce that $f^{-1}(U_{0}) \cap X_{1} = f^{-1}(U_{0} \cap \bigcup_{b \in I_{\tau}}g_{b}(X))  = \emptyset$. %Here, $X_{1} := \bigcup_{b \in I_{\tau}}\phi_{b}(X)$. 
It follows $\inter(X) \setminus X_{1} \supset f^{-1}(U_{0})$. 

Therefore, we deduce that $\lambda_{2}(\inter(X) \setminus X_{1}) > 0$, where $\lambda_{2}$ is the 2-dimensional Lebesgue measure. 
By Proposition 4.4 of \cite{MU}, we obtain that $h_{\tau} < 2$. 
Hence, we have proved $1 < h_{\tau} < 2$. 
\end{proof}

\begin{lemma} \label{Xinftythorem}
Let $\tau \in A_{0}$. 
Then, we have that $X_{\tau}(\infty) = \{0\}$. 
\end{lemma}
\begin{proof}
We first show that for all $\tau \in A_{0}$, $0 \in X_{\tau}(\infty)$. 
We set $I_{\tau}^{\prime} := \{ m + \tau \in I_{\tau} \ |\ m \in \mathbb{N} \} \subset I_{\tau}$ and $b_{m} := m + \tau \in I_{\tau}^{\prime}$. 
Then, we have that $|I_{\tau}^{\prime}| = \infty$ and since $0 \in X$, $\phi_{b_{m}}(0) \in \phi_{b_{m}}(X)$. 
Let $\epsilon > 0$. 
Then, there exists $M \in \mathbb{N}$ such that $M > 1/\epsilon$. 
Let $F_{\tau} := \{ m + \tau \in I_{\tau} \ |\ m \in \mathbb{N}, m \leq M \} \subset I_{\tau}^{\prime}$. 
We obtain that $|F_{\tau}| < \infty$ and if $b_{m} \in I_{\tau}^{\prime} \setminus F_{\tau}$, then $\phi_{b_{m}}(0) \in \phi_{b_{m}}(X)$ and
\[|\phi_{b_{m}}(0)| = \left| \frac{1}{m + \tau} \right| < \frac{1}{m} < \frac{1}{M} < \epsilon. \]

We next show that for each $\tau \in A_{0}$, $a \in X_{\tau}(\infty)$ implies $a = 0$. 
Suppose that there exists $a \in X_{\tau}(\infty)$ such that $a \neq 0$. 
Then, there exist $I^{\prime}_{\tau} \subset I_{\tau}$ and $\{z_{b}^{\prime}\}_{b \in I^{\prime}_{\tau}}$ such that $|I^{\prime}_{\tau}| = \infty$, $z_{b}^{\prime} \in \phi_{b}(X) \ \ (b \in I^{\prime}_{\tau})$ and $\displaystyle \lim_{b \in I^{\prime}_{\tau}}z_{b}^{\prime} = a$. 
Let $\delta := |a|/2 > 0$. 
Then, there exists $F_{\tau}^{\prime} \subset I_{\tau}^{\prime}$ such that $|F_{\tau}^{\prime}| < \infty$ and for all $b \in I^{\prime}_{\tau} \setminus F_{\tau}^{\prime}$, $|z_{b}^{\prime} -a| < \delta$. 
In particular, for all $b \in I^{\prime}_{\tau} \setminus F_{\tau}^{\prime}$, 
\begin{equation} \label{xbabove}
|z_{b}^{\prime}| \geq |a| - |z_{b}^{\prime} - a| > \delta. 
\end{equation}

Moreover, for each $z \in X$, $\tau \in A_{0}$ and $b \in I_{\tau}$, we write $z := x +yi$, $\tau := u+iv$ and $b := m+n\tau$. 
Note that

\begin{align*}
|z+b|^{2}
&= |x+m+nu+i(y+nv)|^{2} 
= (x+m+nu)^{2}+(y+nv)^{2}\\
&\geq (0+m+nu)^{2} + (-1/2+nv)^{2} \geq m^{2} + (n -1/2)^{2}. 
\end{align*}
Let $M := 1/\delta$. 
By using the above inequality, there exists $N_{\delta} \in \mathbb{N}$ such that for all $m \in \mathbb{N}$, $n \in \mathbb{N}$ and $x \in X$, if $m \geq N_{\delta}$ or $n \geq N_{\delta}$, then $|z+b| > M = 1/\delta$. 
In particular, $b \in I_{\tau} \setminus F_{\tau}(N_{\delta})$ implies that for all $z \in X$, $|\phi_{b}(z)|< \delta$. 
Here, $F_{\tau}(N_{\delta}) := \{ b:= m+n\tau \in I_{\tau} \ |\ n \leq N_{\delta}, m \leq N_{\delta} \}$. 

By the inequality (\ref{xbabove}) and $|F_{\tau}(N_{\delta})| < \infty$, this contradicts that there exist $b \in I^{\prime}_{\tau} \setminus (F_{\tau}^{\prime}\cup F_{\tau}(N_{\delta}))$ and $z_{b}^{\prime} \in \phi_{b}(X)$ such that $|z_{b}^{\prime}| > \delta$. 
Therefore, we have proved that for all $\tau \in A_{0}$, $X_{\tau}(\infty) = \{0\}$. 
\end{proof}
\section{Proof of the main result}
In this section, we prove the main result Theorem \ref{main}. 
In order to prove Theorem \ref{main}, we first show a basic estimate for the conformal measure. 

Note that for each $\tau \in A_{0}$, there exists the unique $h_{\tau}$-conformal measure $m_{S_{\tau}}$ of $S_{\tau}$ by Proposition \ref{existenceofconformalmeasure} since for each $\tau \in A_{0}$, $S_{\tau}$ is hereditarily regular. 
We set $m_{\tau} := m_{S_{\tau}}$. 
\begin{lemma} \label{basicconformalestimate}
Let $\tau \in A_{0}$ and $m_{\tau}$ be the $h_{\tau}$-conformal measure of $S_{\tau}$. 
Then, there exists $K_{0} \geq 1$ such that for each $b \in I_{\tau}$, we have $\phi_{b}(X) \subset B(0, K_{0} |b|^{-1})$ and 
\begin{equation*}
m_{\tau}(\phi_{b}(X)) \geq K_{0}^{-h_{\tau}}|b|^{-2h_{\tau}}. 
\end{equation*}
\end{lemma}
\begin{proof}
By Lemma \ref{basicestimate} with $K = K_{0}$, we deduce that for all $b \in I_{\tau}$ and $z \in V$, $ \phi_{b}(V) \subset B(0, K_{0} |b|^{-1}) \quad \text{and} \quad K_{0}^{-1} |b|^{-2} \leq |\phi_{b}^{\prime}(z)| \leq K_{0} |b|^{-2}$. 
Therefore, we have $\phi_{b}(X) \subset B(0, K_{0} |b|^{-1})$ and 
\begin{equation*}
m_{\tau}(\phi_{b}(X)) = \int_{X} |\phi_{b}^{\prime}|^{h_{\tau}} \mathrm{d} m_{\tau} \geq (K_{0}^{-1}|b|^{-2})^{h_{\tau}}m_{\tau}(X) = K_{0}^{-h_{\tau}}|b|^{-2h_{\tau}}. 
\end{equation*}
Thus, we have proved our lemma. 
\end{proof}
We explain the idea of the proof of Theorem \ref{main}. 
Recall that $X_{\tau}(\infty) = \{0\}$. 
By Lemma \ref{basicconformalestimate}, for sufficiently small $r >0$, $b \in I_{\tau}$ and $N>0$ with $r/N < K_{0}|b|^{-1} < r$, we have $\phi_{b}(X) \subset B(0, K_{0} |b|^{-1}) \subset B(0, r)$ and
\begin{equation} \label{basicconformalestimatewithr}
\frac{m_{\tau}(B(0, r))}{r^{h_{\tau}}} \geq \frac{m_{\tau}(\phi_{b}(X))}{r^{h_{\tau}}} \geq \frac{K_{0}^{-h_{\tau}}|b|^{-2h_{\tau}}}{r^{h_{\tau}}} \geq \frac{K_{0}^{-h_{\tau}}}{r^{h_{\tau}}} \left( \frac{r}{NK_{0}} \right)^{2h_{\tau}} \simeq r^{h_{\tau}}. 
\end{equation}
This inequality (\ref{basicconformalestimatewithr}) does not satisfy the assumption of Theorem \ref{Hausdorffmeasureestimate} unfortunately. 
However, since for all $b, b^{\prime} \in I_{\tau}$ with $b \neq b^{\prime}$, $m_{\tau}(\phi_{b}(X) \cap \phi_{b^{\prime}}(X) ) = 0$, we have a sharper estimate on the value of $m_{\tau}(B(0, r))$. 
To obtain this estimate, we set
\[I_{\tau}(r) := \{ b \in I_{\tau} \ | \ r/N_{\tau} \leq K_{0} |b|^{-1} < r \}, \]
where $N_{\tau}$ is the number we introduce later. 
Then, in the proof of Theorem \ref{main}, we will show that  
\begin{equation}\label{intuitiveconformalmeasureestimate}
\frac{m_{\tau}(B(0, r))}{r^{h_{\tau}}} \geq \sum_{b \in I_{\tau}(r)}\frac{m_{\tau}(\phi_{b}(X))}{r^{h_{\tau}}} \geq |I_{\tau}(r)| K_{0}^{-3h_{\tau}}N_{\tau}^{-2h_{\tau}}r^{h_{\tau}}. 
\end{equation}
Note that since $I_{\tau}(r) = \{ b \in I_{\tau} \ | \ K_{0}r^{-1} < |b| \leq N_{\tau} K_{0}r^{-1} \}$, we have
\begin{equation}\label{intuitivelatticepointnumber}
|I_{\tau}(r)| \gtrsim r^{-2}
\end{equation}
intuitively since we have a intuition that the number of the points $b \in I_{\tau}(r)$ in the slant lattice $I_{\tau}$ is almost the same as the area of $I_{\tau}(r)$. 
This estimate will be a key estimate in the proof of Theorem \ref{main}. 
After proving Lemma \ref{insidetwocircle}, Proposition \ref{latticepointlemma} and Lemma \ref{latticepointestimate}, we will rigorously show estimate (\ref{intuitivelatticepointnumber}), whose precise statement is given by (\ref{taulatticepointestimate}) later.  
%Therefore, by inequality(\ref{intuitiveconformalmeasureestimate}) and (\ref{intuitivelatticepointnumber}), we have 
%\begin{equation*}
%\frac{m_{\tau}(B(0, r))}{r^{h_{\tau}}} \gtrsim r^{-2} %K_{0}^{-3h_{\tau}}N_{\tau}^{-2h_{\tau}}r^{h_{\tau}} = %K_{0}^{-3h_{\tau}}N_{\tau}^{-2h_{\tau}}\frac{1}{r^{2-h_{\tau}}} %\longrightarrow \infty. 
%\end{equation*}
%as $r \longrightarrow 0$ and we deduce the assumption of Theorem %\ref{Hausdorffmeasureestimate}. 

To prove this intuitive estimate (\ref{intuitivelatticepointnumber}) rigorously, we introduce the following notations and prove Lemma \ref{insidetwocircle}, Proposition \ref{latticepointlemma} and Lemma \ref{latticepointestimate}. 
We identify $\mathbb{C}$ with $\mathbb{R}^{2}$, $I_{\tau}$ with $\{ {}^{t}(a, b) \in \mathbb{R}^{2} \ | \ a+ib \in I_{\tau} \}$ and $\mathbb{N}^{2}$ with $\{ {}^{t} (m, n) \in \mathbb{R}^{2} \ | \ m, n \in \mathbb{N} \}$, where for any matrix $A$, we denote by ${}^{t} A$ the transpose of $A$.  
For each $\tau = u + iv \in A_{0}$, we set \[ E_{\tau} := \left( \begin{array}{cc} 1 & u \\ 0 & v \end{array}\right) \quad \text{and} \quad F_{\tau} := {}^{t}E_{\tau} E_{\tau} = \left( \begin{array}{cc} 1 & u \\ u & |\tau|^{2} \end{array}\right). \]
Note that $E_{\tau}\mathbb{N}^{2} = I_{\tau}$, since $\det(E_{\tau}) = v \neq 0$, $E_{\tau}$ is invertible and by direct calculations, there exist the eigenvalues $\lambda_{1} > 0$ and $\lambda_{2} > 0$ of $F_{\tau}$ with $\lambda_{1} < \lambda_{2}$. 
Let $v_{1} \in \mathbb{R}^{2}$ be a eigenvector with respect to $\lambda_{1}$ and $v_{2} \in \mathbb{R}^{2}$ be a eigenvector with respect to $\lambda_{2}$. 
Note that since $F_{\tau}$ is a symmetric matrix, there exist eigenvectors $v_{1} \in \mathbb{R}^{2}$ and $v_{2} \in \mathbb{R}^{2}$ such that $V_{\tau} := (v_{1}, v_{2})$ is an orthogonal matrix. 

For each $R_{1} > 0$ and $R_{2} > 0$  with $R_{1}/\sqrt{\lambda_{1}} < R_{2}/\sqrt{\lambda_{2}}$, we set 
\begin{align*}
& D_{1}^{\prime}(\tau, R_{1}, R_{2}) := \{ ^{t}(x, y) \in \mathbb{R}^{2} \ | \ R_{1}^{2}/\lambda_{1} < x^{2} + y^{2} \leq R_{2}^{2}/\lambda_{2} \} \quad \text{and} \\
& D_{2}^{\prime}(R_{1}, R_{2}) := \{ ^{t}(x, y) \in \mathbb{R}^{2} \ |\ R_{1}^{2} < x^{2} + y^{2} \leq R_{2}^{2} \}. 
\end{align*}
We show the following statement on the annuli $D_{1}^{\prime}(\tau, R_{1}, R_{2})$ and $D_{2}^{\prime}(R_{1}, R_{2})$. 
\begin{lemma} \label{insidetwocircle}
Let $\tau \in A_{0}$ and let $R_{1} > 0$ and $R_{2} > 0$ with $R_{1}/\sqrt{\lambda_{1}} < R_{2}/\sqrt{\lambda_{2}}$. 
Then, we have that $E_{\tau}(D_{1}^{\prime}(\tau, R_{1}, R_{2})) \subset D_{2}^{\prime}(R_{1}, R_{2})$. 
In particular, we have that 
\[ E_{\tau} (\mathbb{N}^{2} \cap D_{1}^{\prime}(\tau, R_{1}, R_{2})) \subset I_{\tau} \cap D_{2}^{\prime}(R_{1}, R_{2}) \quad \text{
and} \quad |\mathbb{N}^{2} \cap D_{1}^{\prime}(\tau, R_{1}, R_{2})| \leq |I_{\tau} \cap D_{2}^{\prime}(R_{1}, R_{2})|. \]
\end{lemma}
\begin{proof}
By the above observation of $F_{\tau}$, we deduce that 
%Note that the diagonalization of $F_{\tau}$ is the following. 
\begin{equation*} \label{diagofE}
F_{\tau} = \ V_{\tau} \left( \begin{array}{cc} \lambda_{1} & 0 \\ 0 & \lambda_{2} \end{array}\right) {}^{t}V_{\tau}. 
\end{equation*}

Let $^{t}(x, y) \in D_{1}^{\prime}(\tau, R_{1}, R_{2})$. 
We set $(x^{\prime}, y^{\prime}) := (x, y) \ V_{\tau}$ and $ (v, w) := (x, y) \ ^{t}E_{\tau}$. 
Note that since $V_{\tau}$ is an orthogonal matrix, we deduce that $(x^{\prime})^{2}+(y^{\prime})^{2} = x^{2} + y^{2}$. 
Since $\lambda_{1} < \lambda_{2}$, we have
\begin{align*}
R_{1}^{2} 
&< \lambda_{1}(x^{2} + y^{2}) 
= \lambda_{1}((x^{\prime})^{2}+(y^{\prime})^{2}) 
< \lambda_{1}(x^{\prime})^{2} + \lambda_{2}(y^{\prime})^{2} \\
&= (x^{\prime}, y^{\prime}) \left( \begin{array}{cc} \lambda_{1} & 0 \\ 0 & \lambda_{2} \end{array}\right) \left( \begin{array}{c} x^{\prime} \\ y^{\prime}  \end{array}\right) 
= (x, y) \ V_{\tau} \left( \begin{array}{cc} \lambda_{1} & 0 \\ 0 & \lambda_{2} \end{array}\right) {}^{t}V_{\tau} \left( \begin{array}{c} x \\ y  \end{array}\right) \\
&= (x, y) \ F_{\tau} \left( \begin{array}{c} x \\ y  \end{array}\right) 
= (x, y) \ ^{t}E_{\tau} E_{\tau} \left( \begin{array}{c} x \\ y  \end{array}\right). 
\end{align*}
By the above inequality, we deduce that $R_{1}^{2} < v^{2} + w^{2}$. 
Also, 
\begin{align*}
R_{2}^{2} 
&\geq \lambda_{2}(x^{2} + y^{2}) 
= \lambda_{2}((x^{\prime})^{2}+(y^{\prime})^{2}) 
\geq \lambda_{1}(x^{\prime})^{2} + \lambda_{2}(y^{\prime})^{2} \\
&= (x^{\prime}, y^{\prime}) \left( \begin{array}{cc} \lambda_{1} & 0 \\ 0 & \lambda_{2} \end{array}\right) \left( \begin{array}{c} x^{\prime} \\ y^{\prime}  \end{array}\right) 
= (x, y) \ V_{\tau} \left( \begin{array}{cc} \lambda_{1} & 0 \\ 0 & \lambda_{2} \end{array}\right) {}^{t}V_{\tau} \left( \begin{array}{c} x \\ y  \end{array}\right) \\
&= (x, y) \ F_{\tau} \left( \begin{array}{c} x \\ y  \end{array}\right) 
= (x, y) \ ^{t}E_{\tau} E_{\tau} \left( \begin{array}{c} x \\ y  \end{array}\right). 
\end{align*}
By the above inequality, we deduce that $v^{2} + w^{2} \leq R_{2}^{2}$. 
Therefore, we have proved our lemma. 
\end{proof}
For each $R > 0$, we set $I(R) := \{ ^{t}(m, n) \in \mathbb{N}^{2} \ | \ m^{2} + n^{2} \leq R^{2} \}$. 

We give the following estimate on $|I(R)|$. 
\begin{proposition} \label{latticepointlemma}
Let $R > 0$. Then, for each $R \geq 6$, 
\[ 0 < \frac{R^{2}-7R+7}{2} \leq |I(R)| \leq R^{2}. \]
\end{proposition}
\begin{proof}
For each $a \in \mathbb{R}$, we denote by $\lfloor a \rfloor$ the maximum integer of the set $\{ n \in \mathbb{Z} \ | \ n \leq a \}$. 
Let $R \geq 6$. 
We set $M:= \lfloor \sqrt{R^{2} -1} \rfloor (\geq 1)$. 
For each $l = 1, \ldots , M$, we set $N(l):= \lfloor \sqrt{R^{2} - l^{2}} \rfloor (\geq 1)$. 
Note that since $M \leq \sqrt{R^{2} -1} < M+1 $, we deduce that 
\begin{equation} \label{Mestimate}
\sqrt{R^{2} -1} - 1 < M \leq \sqrt{R^{2} -1}. 
\end{equation}
Also, since $N(l) \leq \sqrt{R^{2} - l^{2}} < N(l)+1 $, we deduce that 
\begin{equation} \label{Nestimate}
\sqrt{R^{2} -l^{2}} - 1 < N(l) \leq \sqrt{R^{2} - l^{2}}. 
\end{equation}
Therefore, we deduce that $|I(R)| = \sum_{l = 1}^{M} N(l)$. 

By the inequalities (\ref{Mestimate}) and (\ref{Nestimate}), we deduce that 
\begin{equation*}
|I(R)| \leq \sum_{l = 1}^{M} \sqrt{R^{2} - l^{2}} \leq R M \leq R \sqrt{R^{2} -1} \leq R^{2}. 
\end{equation*}

We now show that $|I(R)| \geq (R^{2} -7R +7)/2$. 
Since $\sqrt{R^{2} - l^{2}} \geq R - l$ for each $l = 1, \ldots , M$, by the inequalities (\ref{Mestimate}) and (\ref{Nestimate}) again, we deduce that 
\begin{align*}
|I(R)| 
&\geq \sum_{l = 1}^{M} \left(\sqrt{R^{2} - l^{2}}-1\right) 
\geq \sum_{l = 1}^{M} (R - l - 1) 
= M(R-1) - \frac{M(M+1)}{2} \\
&= \frac{M(2R-3) - M^{2}}{2} 
\geq \frac{\left(\sqrt{R^{2}-1}-1\right)(2R-3) - (R^{2}-1)}{2} \\
&\geq \frac{(R-2)(2R-3) -R^{2} + 1}{2} 
= \frac{R^{2}-7R+7}{2}. 
\end{align*}
Thus, we have proved our lemma.  
\end{proof}
For each $\tau \in A_{0}$, we set $N_{\tau} := \sqrt{2\lambda_{2}}/\sqrt{\lambda_{1}} + 1 \ (> 2)$. 
For each $R > 0$, we set $D_{1}(\tau, R) := D_{1}^{\prime}(\tau, R, N_{\tau}R)$ and $D_{2}(\tau, R) := D_{2}^{\prime}(R, N_{\tau}R)$. 
Note that since $\sqrt{\lambda_{2}}/\sqrt{\lambda_{1}} < N_{\tau}$, we have that $R/\sqrt{\lambda_{1}} < (N_{\tau}R)/\sqrt{\lambda_{2}}$. 

We estimate $|\mathbb{N}^{2} \cap D_{1}(\tau, R)|$ from below as follows. 
\begin{lemma} \label{latticepointestimate}
Let $\tau \in A_{0}$. Then, there exist $R_{\tau} > 0$ and $L_{\tau}>0$ such that for all $R > R_{\tau}$, 
\[|\mathbb{N}^{2} \cap D_{1}(\tau, R)| \geq L_{\tau} R^{2} -  \frac{7N_{\tau}}{2\sqrt{\lambda_{2}}} R. \]
\end{lemma}
\begin{proof}
Let $\tau \in A_{0}$. 
We set $L_{\tau} := N_{\tau}^{2}/(2\lambda_{2})-1/\lambda_{1}$. 
Note that since $N_{\tau} > \sqrt{2\lambda_{2}}/\sqrt{\lambda_{1}}$, we deduce that $L_{\tau} > 0$. 
We set 
\[R_{\tau} := \max\{(6\sqrt{\lambda_{2}})/N_{\tau}, 6\sqrt{\lambda_{1}}\} (> 0). \]
Let $R \geq R_{\tau}$. 
Note that $N_{\tau}R/\sqrt{\lambda_{2}} \geq 6$, $R/\sqrt{\lambda_{1}} \geq 6$ and 
\begin{equation} \label{circlesetminus}
\mathbb{N}^{2} \cap D_{1}(\tau, R) = I\left(\frac{N_{\tau}R}{\sqrt{\lambda_{2}}}\right) \setminus I\left(\frac{R}{\sqrt{\lambda_{1}}}\right). 
\end{equation}
Also, we have $I\left((N_{\tau}R)/\sqrt{\lambda_{2}}\right) \supset I\left(R/\sqrt{\lambda_{1}}\right)$. 
By (\ref{circlesetminus}) and Proposition \ref{latticepointlemma}, we deduce that 
\begin{align*}
|\mathbb{N}^{2} \cap D_{1}(\tau, R)|
&= \left|I\left(\frac{N_{\tau}R}{\sqrt{\lambda_{2}}}\right)\right| - \left|I\left(\frac{R}{\sqrt{\lambda_{1}}}\right)\right| \\
&\geq \frac{1}{2}\left( \frac{(N_{\tau}R)^{2}}{\lambda_{2}} - 7 \frac{N_{\tau}R}{\sqrt{\lambda_{2}}} + 7 \right) - \frac{R^{2}}{\lambda_{1}} 
> L_{\tau}R^{2} -  \frac{7N_{\tau}}{2\sqrt{\lambda_{2}}} R. 
\end{align*}
Therefore, we have proved our lemma. 
\end{proof}

By Lemma \ref{insidetwocircle}, Proposition \ref{latticepointlemma} and Lemma \ref{latticepointestimate}, we now prove the intuitive estimate (\ref{intuitivelatticepointnumber}) rigorously. 
\begin{proof}[Rigorous proof of the estimate  (\ref{intuitivelatticepointnumber})]
Let $\tau \in A_{0}$. 
We set $r_{\tau} := K_{0} R_{\tau}^{-1} (> 0)$ and $M_{\tau} := (7N_{\tau})/(2\sqrt{\lambda_{2}})$. 
We show that for all $r \in (0, r_{\tau}]$, 
\begin{equation} \label{taulatticepointestimate}
|I_{\tau}(r)| = |\{ b \in I_{\tau} \ | \ r/N_{\tau} \leq K_{0} |b|^{-1} < r \}|  \geq L_{\tau} K_{0}^{2} r^{-2} - M_{\tau} K_{0} r^{-1}. 
\end{equation}
Let $r \in (0, r_{\tau}]$. 
We set $R := K_{0} r^{-1}$. 
Note that $r \leq r_{\tau}$ if and only if $R \geq R_{\tau}$. 
Recall that $I_{\tau}(r) := \{ b \in I_{\tau} \ | \ r/N_{\tau} \leq K_{0} |b|^{-1} < r \}$ and 
\[I_{\tau}(r) = \{ b \in I_{\tau} \ | \ K_{0}r^{-1} < |b| \leq N_{\tau}K_{0}r^{-1} \} = I_{\tau} \cap D_{2}^{\prime}(K_{0}r^{-1}, N_{\tau}K_{0}r^{-1}). \]
Recall that $R := K_{0} r^{-1}$, $D_{1}(\tau, R) := D_{1}^{\prime}(\tau, R, N_{\tau}R)$ and $M_{\tau} := (7N_{\tau})/(2\sqrt{\lambda_{2}})$. 
By Lemmas \ref{insidetwocircle} and \ref{latticepointestimate}, it follows that 
\begin{align*}
|I_{\tau}(r)| 
&= |I_{\tau} \cap D_{2}^{\prime}(K_{0}r^{-1}, N_{\tau}K_{0}r^{-1})| 
= |I_{\tau} \cap D_{2}^{\prime}(R, N_{\tau}R)|
\geq |\mathbb{N}^{2} \cap D_{1}^{\prime}(\tau, R, N_{\tau}R)| \\
&= |\mathbb{N}^{2} \cap D_{1}(\tau, R)| 
\geq L_{\tau} R^{2} -  M_{\tau} R 
= L_{\tau} K_{0}^{2} r^{-2} - M_{\tau} K_{0} r^{-1}. 
\end{align*}
Thus, we have proved the inequality (\ref{taulatticepointestimate}). 
\end{proof}
We now give the proof of the main result Theorem \ref{main}. 
\begin{proof}[Proof of Theorem \ref{main}]
Let $\tau \in A_{0}$. 
Recall that there exists the unique $h_{S_{\tau}}$-conformal measure $m_{\tau}$ of $S_{\tau}$. 
We set $r_{\tau} := K_{0} R_{\tau}^{-1} (> 0)$ and $M_{\tau} := (7N_{\tau})/(2\sqrt{\lambda_{2}})$. 

We first show that for all $r \in (0, r_{\tau}]$, 
\begin{equation} \label{conformalmeasureestimate}
m_{\tau}(B(0, r)) \geq L_{\tau} K_{0}^{2-3h_{\tau}} N_{\tau}^{-2h_{\tau}} r^{2h_{\tau}-2} - M_{\tau} K_{0}^{1-3h_{\tau}} N_{\tau}^{-2h_{\tau}} r^{2h_{\tau}-1}. 
\end{equation}
By Lemma \ref{basicconformalestimate} and the definition of $I_{\tau}(r)$, we have that for all $b \in I_{\tau}(r)$, 
$\phi_{b}(X) \subset B(0, K_{0} |b|^{-1}) \subset B(0, r)$. 
It follows that
\begin{equation} \label{conformalmappinginclusion}
\bigcup_{b \in I_{\tau}(r)} \phi_{b}(X) \subset B(0, r). 
\end{equation}
In addition, if $b, b^{\prime} \in  I_{\tau}$ with $b \neq b^{\prime}$, then $m_{\tau}(\phi_{b}(X) \cap \phi_{b^{\prime}}(X) ) = 0$ by the definition of the conformal measure (Proposition \ref{existenceofconformalmeasure}).
Thus, by inclusion (\ref{conformalmappinginclusion}) and Lemma \ref{basicconformalestimate}, it follows that 
\begin{align*}
m_{\tau}(B(0, r)) &\geq m_{\tau}\left(\bigcup_{b \in I_{\tau}(r)} \phi_{b}(X)\right) 
= \sum_{b \in I_{\tau}(r)} m_{\tau}(\phi_{b}(X)) \geq \sum_{b \in I_{\tau}(r)} K_{0}^{-h_{\tau}}|b|^{-2h_{\tau}} \\
&\geq \sum_{b \in I_{\tau}(r)} K_{0}^{-h_{\tau}} \left( \frac{r}{N_{\tau}K_{0}} \right)^{2h_{\tau}} = |I_{\tau}(r)| K_{0}^{-3h_{\tau}} N_{\tau}^{-2h_{\tau}} r^{2h_{\tau}}. 
\end{align*}
By the inequality (\ref{taulatticepointestimate}), we obtain that 
\[ m_{\tau}(B(0, r)) \geq L_{\tau} K_{0}^{2-3h_{\tau}} N_{\tau}^{-2h_{\tau}} r^{2h_{\tau}-2} - M_{\tau} K_{0}^{1-3h_{\tau}} N_{\tau}^{-2h_{\tau}} r^{2h_{\tau}-1}. \]
Thus, we have proved inequality (\ref{conformalmeasureestimate}). 

We now show that $\mathcal{H}^{h_{\tau}}(J_{\tau}) = 0$. 
For each $j \in \mathbb{N}$, we set $z_{j} := 0$ and $r_{j} := r_{\tau}/j \ (\in (0, r_{\tau}])$. 
Note that $\{r_{j}\}_{j \in \mathbb{N}}$ is a sequence in the set of positive real numbers and by Lemma \ref{Xinftythorem},
$\{z_{j}\}_{j \in \mathbb{N}}$ is a sequence in $X_{\tau}(\infty)$. 
Thus, by the inequality (\ref{conformalmeasureestimate}), we deduce that for each $j \in \mathbb{N}$, 
\begin{align*}
&\frac{m_{\tau}(B(z_{j}, r_{j}))}{r_{j}^{h_{\tau}}} 
= \frac{m_{\tau}(B(0, r_{j}))}{r_{j}^{h_{\tau}}} \\
&\geq L_{\tau} K_{0}^{2-3h_{\tau}} N_{\tau}^{-2h_{\tau}} r_{j}^{h_{\tau}-2} - M_{\tau} K_{0}^{1-3h_{\tau}} N_{\tau}^{-2h_{\tau}} r_{j}^{h_{\tau}-1} \\
&= L_{\tau} K_{0}^{2-3h_{\tau}} N_{\tau}^{-2h_{\tau}} r_{\tau}^{h_{\tau}-2} j^{2-h_{\tau}} - M_{\tau} K_{0}^{1-3h_{\tau}} N_{\tau}^{-2h_{\tau}} r_{\tau}^{h_{\tau}-1} \left(\frac{1}{j}\right)^{h_{\tau}-1}.
\end{align*}
By Lemma \ref{betweenoneandtwo}, we have that $2-h_{\tau} > 0$ and $h_{\tau} - 1 > 0$. It follows that
\[ \limsup_{j \to \infty} \frac{m_{\tau}(B(z_{j}, r_{j}))}{r_{j}^{h_{\tau}}} = \infty. \]
By Theorem \ref{Hausdorffmeasureestimate}, we obtain that $\mathcal{H}^{h_{\tau}}(J_{\tau}) = 0$. 

We finally show that $\mathcal{P}^{h_{\tau}}(J_{\tau}) > 0$. 
Let $\tau = u + iv \in A_{0}$. 
We set $b_{2} := 2 + \tau \in I_{\tau}$. 
We use some notations in Lemma \ref{GCCFisCIFS}. 
For any $z = x + iy \in X$, 
\begin{align*}
g_{b_{2}}(z) &= z +(2 +\tau) \\
&= (x + 2 + u) + i(y + v) \in \{z \in \mathbb{C} \ | \ \Re z > 1 \} = \inter(Y). 
\end{align*}
Since $f(\partial X) = \partial Y \cup \{ \infty \}$ and $f \colon X \to Y \cup \{ \infty \}$ is bijective, we have 
\[\phi_{b_{2}}(X) = (f^{-1}\circ g_{b_{2}})(X) \subset \inter(X). \]
Therefore, we obtain that $J_{\tau} \cap \inter(X) \neq \emptyset$. 
Since $S_{\tau}$ is hereditarily regular and $J_{\tau} \cap \inter(X) \neq \emptyset$, we deduce that $\mathcal{P}^{h_{\tau}}(J_{\tau}) > 0$ by Theorem \ref{packingmeasureestimate}. 
\end{proof}

\end{document}